\documentclass[12pt]{amsart}
\usepackage{amsmath,amssymb,amsbsy,amsfonts,latexsym,amsopn,amstext,cite,
                                               amsxtra,euscript,amscd,bm,mathabx,mathrsfs}
\usepackage{url}
\usepackage[colorlinks,linkcolor=blue,anchorcolor=blue,citecolor=blue,backref=page]{hyperref}
\usepackage{color}
\usepackage{graphics,epsfig}
\usepackage{graphicx}
\usepackage{float} 
\usepackage[english]{babel}
\usepackage{mathtools}
\usepackage{todonotes}
\usepackage{url}
\usepackage[colorlinks,linkcolor=blue,anchorcolor=blue,citecolor=blue,backref=page]{hyperref}

\usepackage[norefs,nocites]{refcheck}

\hypersetup{breaklinks=true}

\usepackage[norefs,nocites]{refcheck}
\usepackage[english]{babel}

\newtheorem{thm}{Theorem}
\newtheorem{lem}[thm]{Lemma}

\newtheorem{cor}[thm]{Corollary}

\newtheorem{conj}[thm]{Conjecture}

\numberwithin{equation}{section}
\numberwithin{thm}{section}
\numberwithin{table}{section}

\def\squareforqed{\hbox{\rlap{$\sqcap$}$\sqcup$}}
\def\qed{\ifmmode\squareforqed\else{\unskip\nobreak\hfil
\penalty50\hskip1em\null\nobreak\hfil\squareforqed
\parfillskip=0pt\finalhyphendemerits=0\endgraf}\fi}

\def\sqrt{\mathrm{sqrt}}

\def \F {{\mathbb F}}

\def \Z {{\mathbb Z}}

\def\\{\cr}
\def\({\left(}
\def\){\right)}
\def\fl#1{\left\lfloor#1\right\rfloor}

\newcommand{\ov}{\overbar}
\newcommand{\overbar}[1]{\mkern 1.5mu\overline{\mkern-1.5mu#1\mkern-1.5mu}\mkern 1.5mu}

 \newcommand{\Mod}[1]{\ (\mathrm{mod}\ #1)}

 \newcommand{\slider}{11/20}

\begin{document}

\title[Square-free smooth polynomials in residue classes  ]
{Square-free smooth polynomials in residue classes and generators of irreducible polynomials}

 \author[C.~Bagshaw]{Christian Bagshaw}
 \address{School of Mathematics and Statistics, University of New South Wales.
 Sydney, NSW 2052, Australia}
 \email{c.bagshaw@unsw.edu.au}

\begin{abstract}  
Building upon the work of A. Booker and C. Pomerance (2017), we prove that for a prime power $q \geq 7$, every residue class modulo an irreducible polynomial $F \in \mathbb{F}_q[X]$ has a non-constant, square-free representative which has no irreducible factors of degree exceeding $\deg F -1$. We also give applications to generating sequences of irreducible polynomials. 
\end{abstract}  


\maketitle

\tableofcontents

\section{Introduction}

\subsection{Motivation}
We recall that an integer $n$ is {\it $k$-smooth} if no prime divisor of $n$ exceeds $k$, and $n$ is {\it square-free} if it is not divisible by the square of a prime. 

For a prime $p$ we denote by $M(p)$ the smallest integer such that any residue class modulo $p$ can be represented by a $p$-smooth, square-free representative not exceeding $M(p)$, and we formally set $M(p) = \infty$ if no such value exists. Booker and Pomerance \cite{BP2017} prove that $M(p) < \infty$ for $p \geq 11$, give the bound $M(p) = p^{O(\log p)}$ and conjecture $M(p) \leq p^{O(1)}$. This conjecture has been settled in a more general setting in \cite{MS2020}, where it is shown that $M(p) = p^{3/2 + o(1)}$. 

The quantity $M(p)$ was initially studied for its applications to recursive prime generators, which take roots in Euclid's proof of the infinitude of primes. The Euclid-Mullin sequence $\{p_k\}$ is defined such that for every $k \geq 0$, $p_{k+1}$ is the smallest prime factor of $p_1...p_k + 1$. In \cite{M1963}, Mullin asks whether {\it every} prime occurs in this sequence. This remains an open question, but variations of this sequence have been investigated. Booker and Pomerance \cite{BP2017} consider two such variations, and use $M(p) < \infty$ for $p \geq 11$ to show that they each contain every prime. 

Here we obtain analogues of these results for polynomials over finite fields. Although our general approach is similar to that of Booker and Pomerance~\cite{BP2017}, we take advantage of very strong bounds of short character sums (see Lemma~\ref{lem:MonicSum} below) which are not available over the integers. On the other hand, the distinction between monic and non-monic polynomials presents some new technical challenges. 
\subsection{Main results} Given a prime power $q$, we denote by $\F_q$ the finite field of order $q$. We call a polynomial $f(X) \in \F_q[X]$ {\it $k$-smooth} if $f$ has no irreducible factors of degree exceeding $k$, and we say that $f$ is {\it square-free} if it is not divisible by the square of an irreducible polynomial. 

For an irreducible polynomial $F(X) \in \F_q[X]$ of degree $r$, we denote by $M_q(F)$ the smallest integer such that any non-zero residue class modulo $F$ has an $r$-smooth, square-free representative whose degree does not exceed $M_q(F)$. Again, we formally set $M_q(F) = \infty$ if no such value exists. In \cite{BS2021} it is shown that $M_q(F) = (2+o(1))r$ as $r \to \infty$, but it is not shown exactly for which $F$ we have $M_q(F) < \infty$. Here we prove the following:

\begin{thm}\label{thm:SqFreeSmooth}
Let $q \geq 7$ be a prime power, and let $F(X) \in \F_q[X]$ be an irreducible polynomial of degree $r \geq 2$. Then every non-zero residue class modulo $F$ has a non-constant, square-free, $r-1$-smooth representative.
\end{thm}

Of course if we drop the condition that the representative is to be non-constant, then the condition $r \geq 2$ can be changed to $r \geq 1$. We also conjecture the following:
\begin{conj}\label{conj:allq}
Theorem~\ref{thm:SqFreeSmooth} holds for any prime power $q\geq 3$, and for $q =2$ as long as $r \geq 4$. 
\end{conj}
Subsection \ref{subsec:conjecture} contains comments on this conjecture. 

In direct analogy to \cite[Section 5]{BP2017}, we can apply Theorem \ref{thm:SqFreeSmooth} to generate sequences of irreducible polyomials. In the following two results, the conditions $q \geq 7$ and $p \geq 7$ can be changed to $q \geq 3$ and $p \geq 3$ if Conjecture \ref{conj:allq} is settled. 

\begin{cor}\label{cor:plus_n/d_sequence}
Let $q \geq 7$ be an odd prime power. Starting with all degree $1$ polynomials as the first terms in the sequence, recursively define a sequence of irreducible polynomials in $\F_q[X]$ where if $N$ is the product of the terms generated so far, take as the next term some prime factor of $h + N/h$ where $h|N$. Such a sequence can be chosen to contain every irreducible polynomial. 
\end{cor}

Now for a prime $p$ we introduce a total order on $\F_p[X]$ which we take from \cite[Definition 2.4]{KS2008}. For $f(X) := \sum_{k=0}^ma_kX^k \in \F_p[X]$ we set $$I(f) := \sum_{k=0}^m a_kp^k.$$ 
We define a total order on $\F_p[X]$ by $f_1 \leq f_2$ if $I(f_1) \leq I(f_2)$. We prove the following:

\begin{cor}\label{cor:plus_1_sequence}
Let $p \geq 7$ be a prime. Starting with $X$ as the first term in the sequence, recursively define a sequence of monic irreducible polynomials in $\F_p[X]$ where if $N$ is the product of the terms generated so far take 
$$F = \min \{g~\text{monic and irreducible}~:~ g \nmid N,~ g|h+1~\text{for some}~h|N\} $$
as the next term in the sequence. This sequence produces all monic irreducible polynomials in $\F_p[X]$ in order. 
\end{cor}

\subsection{Cases to consider} 
Before proceeding, we define the following subsets of $\mathbb{N} \times \mathbb{N}$, where $q$ is always a prime power:
\begin{align*}
\Omega_0 =& \{(7, r):r \geq 6\} \cup \{(8, r):r \geq 6 \}\cup \{(9, r):r \geq 5 \}\\
&\cup \{(11, r):r \geq 5 \} \cup\{(q,r):13 \leq q < 23, r \geq 4\}\\
&\cup\{(q,r):23 \leq q < 64, r \geq 3\}\cup\{(q,r):q \geq 64, r \geq 2\}
\end{align*}
and
\begin{align*}
\Omega_1 =& \{(q,r): q \geq 7, r \geq 2\} \setminus \Omega_0.
\end{align*}
From this point forward, we say ``$F$ such that $(q,r) \in \Omega_i$" to mean an irreducible polynomial $F(X) \in \F_q[X]$ of degree $r$ such that $(q,r) \in \Omega_i$. 

To prove Theorem \ref{thm:SqFreeSmooth} we must show it's conclusion holds for any $F$ such that $(q,r) \in \Omega_0\cup\Omega_1$. We have split these pairs into two distinct sets, as each set will be dealt with differently.  

\section{Preparations}
\subsection{Character sums and counts for polynomials}
We firstly need a few results regarding square-free polynomials in $\F_q[X]$. It is  convenient to introduce an analogue of  the classical M{\"o}bius function 
$\mu$ for polynomials in $\F_q[X]$:
$$
\mu_q(g) 
=
\begin{cases}
(-1)^k, &g \text{ is square-free and a product of $k$ distinct}\\& \text{irreducible factors,}\\
0, &\text{otherwise}.
\end{cases}
$$

The following is a classical result which is well known in the literature. For example, see \cite[Equation (1-19)]{KR2015}. 
\begin{lem}
For any integer $m \geq 2$, there are $(q-1)(q^m - q^{m-1})$ square-free polynomials of degree exactly $m$. 
\end{lem}
We then have a simple corollary:
\begin{cor}\label{cor:SqFreeCount}
For $m \geq 2$,
$$
\sum_{\substack{\deg f \in [1, m)}} \mu^2_q(f) = 
(q-1)q^{m-1}.
$$
\end{cor}

Next, for any positive integer $t$ let $A_t$ denote the set of all monic polynomials of degree exactly $t$ in $\F_q[X]$. Also, let $F(X) \in \F_q[X]$ be irreducible of degree $r \geq 1$ and let $\chi$ be a non-principal character modulo $F$. The following is given in~\cite[Theorem~1.3]{H2020} (related bounds can also be found in \cite[Theorem 1]{BLL2017}). 

\begin{lem}\label{lem:MonicSum}
For any positive integer $t$, 
$$\left| \sum_{f \in A_t}\chi(f)\right| \leq q^{t/2}\binom{r-1}{t}. $$
\end{lem}
This leads to the following:

\begin{cor}\label{cor:SqFreeSum}
For any positive integer $m \geq 2$ we have 
$$\bigg{|}\sum_{\substack{\deg f \in[1, m)}}\mu_q^2(f)\chi(f) \bigg{|}\leq q^{(m-1)/2}(q-1)(2^{r-1}-1)\frac{m}{2}.$$
\end{cor}
\begin{proof}
We let $S$ denote the sum in question. Firstly, standard inclusion-exclusion gives 
\begin{align*}
    S &= \bigg{|}\sum_{\substack{g \in \F_q[X] \\ g~\text{monic}}}\mu_q(g)\sum_{\substack{\deg f \in[1,m) \\ g^2 |f}}\chi(f)\bigg{|}\\
    &= \bigg{|}\sum_{\substack{g \in \F_q[X] \\ g~\text{monic}}}\mu_q(g)\chi(g^2)\sum_{\deg h \in [1, m-2\deg g)}\chi(h)\bigg{|}.
\end{align*}

Now, we can factor out the leading coefficients on polynomials in the innermost sum and apply the triangle inequality to obtain
\begin{align*}
   S 
   &\leq (q-1)\sum_{\substack{g \in \F_q[X] \\ g~\text{monic}}}\bigg{|}\sum_{\substack{\deg h \in [1, m-2\deg g)\\h~\text{monic}}}\chi(h)\bigg{|}.
\end{align*}
The innermost sum is empty unless $m - 2\deg g \geq 2$, so we have
\begin{align*}
   S 
   &\leq (q-1) \sum_{k=0}^{\fl{\frac{m-2}{2}}}\sum_{\substack{g \in A_k}}\sum_{t=1}^{m-2k - 1}\bigg{|}\sum_{\substack{h \in A_t}}\chi(h) \bigg{|}.
\end{align*}
We note that trivially, 
$$\bigg{|}\sum_{\substack{h \in A_t}}\chi(h) \bigg{|} \leq q^t. $$

Combining this with Lemma \ref{lem:MonicSum} we obtain
\begin{align}\label{eq:min_sqfreesum}
   S 
   &\leq(q-1) \sum_{k=0}^{\fl{\frac{m-2}{2}}}q^k\sum_{t=1}^{m-2 k - 1}q^{t/2}\min\bigg{\{}q^{t/2},\binom{r-1}{t}\bigg{\}}
\end{align}
which will be used later. Further noting that $2^{r-1} \geq \sum_{t=0}^{m-2k-1}\binom{r-1}{t}$ for any $m,k$ we obtain 
$$S \leq (q-1)q^{(m-1)/2}(2^{r-1}-1)\sum_{k=0}^{\fl{\frac{m-2}{2}}}1 \leq (q-1)q^{(m-1)/2}(2^{r-1}-1)\frac{m}{2} $$
which completes the proof. 
\end{proof}

For a positive integer $k\geq 1$ we let $\pi_q(k)$ denote the number of irreducible polynomials of degree exactly $k$ over $\F_q$. If $\mu$ denotes the classical M{\"o}bius function, it is well known that 
\begin{align*}
  \pi_q(k) = \frac{q-1}{k}\sum_{d|k}\mu(d)q^{k/d}  
\end{align*}

which implies 
\begin{align}\label{eq:irred_ineq}
    \frac{q^k-2q^{k/2}}{k} \leq \frac{\pi_q(k)}{q-1} \leq \frac{q^k}{k}.
\end{align}

The following is a technical result, purely used as an aid in proving Theorem \ref{thm:SqFreeSmooth}. 

\begin{lem}\label{lem:IrredCount}
For any pair $(q,r) \in \Omega_0$, 
\begin{align}\label{eq:lem_IrredCount_ineq}
    \sum_{k=1}^{r-1}\pi_q(k) > \frac{q^r}{r}+2q^{r/2 + \slider} + (q-1)r^3+3.
\end{align}
\end{lem}
\begin{proof}
We will proceed by induction on $r$, but firstly we define the following subset of $\Omega_0$:
\begin{align*}
    \Omega_0' =& \{(7, 6), (8,6), (9,5), (11,5)\} \cup\{(q,4):13 \leq q < 23\}
\\&\cup\{(q,3):23 \leq q < 64\}\cup\{(q,2):q \geq 64\}.
\end{align*}

 For a given prime power $q \geq 7$, the base case in the induction will be the unique $r$ such that $(q,r) \in \Omega_0'$. There are 22 pairs $(q,r) \in \Omega_0'$ with $q < 64$, and these can easily be manually checked to satisfy (\ref{eq:lem_IrredCount_ineq}). For $q \geq 64$ we show that the pair $(q,2)$ satisfies (\ref{eq:lem_IrredCount_ineq}). There are exactly $q^2-q$ (irreducible) polynomials of degree $1$, so we have 
 \begin{align*}
      &\sum_{k=1}^{1}\pi_q(k) - \frac{q^2}{2}-2q^{2/2 + \slider} - (q-1)2^3-3 \\
      &= q^2-q - \frac{q^2}{2}-2q^{2/2 + \slider} - (q-1)2^3-3
 \end{align*}
 which is greater than $0$ for $q \geq 53$, but we only need this for $q \geq 64$. 
 
 Suppose for a given pair $(q,r) \in \Omega_0\setminus \Omega_0'$, (\ref{eq:lem_IrredCount_ineq}) holds for the pair $(q,r-1)$. Then using this and (\ref{eq:irred_ineq}) we have 
 \begin{align*}
      &\sum_{k=1}^{r-1}\pi_q(k) - \frac{q^r}{r}-2q^{r/2 + \slider} - (q-1)r^3-3\\
      &\geq \frac{q^{r-1}}{r-1}+2q^{(r-1)/2 + \slider} + (q-1)(r-1)^3 + 3 + (q-1)\frac{q^r-2q^{r/2}}{r} \\
      &\quad\quad - \frac{q^r}{r}-2q^{r/2 + \slider} - (q-1)r^3-3.
 \end{align*}
 It is a routine exercise in calculus to show that this expression is greater than $0$ for any $(q,r)$ with $q \geq 7$ and $r \geq 2$, which more than covers $(q,r) \in \Omega_0\setminus\Omega_0'$ as desired. 
 \end{proof}

\subsection{Square-free  representatives}\label{subsec:SqFreeCoset}
In this section we follow closely the ideas presented by Booker and Pomerance in \cite[Sections 2 and 3]{BP2017}. For an irreducible polynomial $F(X) \in \F_q[X]$ of degree $r$ we can naturally identify $\F_q[X]/F(X) \cong \F_{q^r}$. With this in mind, for any positive integer $d$ dividing $q^r-1$ we define the subgroup of $\F_{q^r}^*$ 
$$H_{d,F} = \{h \in \F_{q^r}^* ~:~h^{(q^r-1)/d} =1\}. $$
Of course this is the unique subgroup of $\F_{q^r}^*$ of index $d$. 
\begin{lem}
Let $f,a \in \F_{q^r}^*$ and let $\chi$ be a character of $\F_{q^r}^*$ of order $d$. Then $\chi(f) = \chi(a)$ if and only if $f \in aH_{d,F}$. 
\end{lem}
\begin{proof}
If $f \in aH_{d,f}$ then $f = ah$ for some $h$ such that $h^{(q^r-1)/d} = 1$. Of course this gives $\chi(f) = \chi(a)\chi(h)$. We wish to show $\chi(h) = 1$. Let $g$ generate $\F_{q^r}^*$. Then there exists some $k$ such that $g^k = h$. This gives $g^{k(q^r-1)/d} = h^{(q^r-1)/d} = 1$, which means that $d | k$, so $k = d\ell$ for some positive integer $\ell$. Thus 
$$\chi(h) = \chi(g^k) = \chi(g^\ell)^d = 1 $$
since $\chi$ has order $d$.

Conversely, suppose $\chi(f) = \chi(a)$. Then since $\chi(fa^{-1}) = 1$, it suffices to show for any $h \in \F_{q^r}^*$, $\chi(h) = 1$ implies $h \in H_{d,F}$. Of course, if $\chi(h) = 1$ then $h \in \ker(\chi)$. By the first isomorphism theorem we have 
$$d = |G|/|\ker(\chi)|. $$
So $|\ker(\chi)|$ is the unique subgroup of $\F_{q^r}^*$ of order $|G|/d$, which is exactly $H_{d,F}$. 
\end{proof}

This then leads to the following:
\begin{lem}\label{lem:coset_rep}
Let $F$ be a polynomial such that $(q,r)\in \Omega_0$. Let $d$ be a divisor of $q^r-1$ with $d < r$. Then for every $a \in \F_{q^r}^*$ there exists some square-free $f$ with $1 \leq \deg f < r$ and $f \in aH_{d, F}$.   
\end{lem}
\begin{proof}
We can assume $d > 1$, since otherwise we can let $f = X$. Let $\chi$ denote a character of $\F_{q^r}^*$ of order $d$. Now consider the expression
$$\sum_{i=1}^d\mu_q^2(f)\chi^i(f)\ov{\chi}^i(a) $$
for some non-zero $f \in \F_q[X]$ with $\deg f < r$. We have three cases to consider. Firstly, if $f$ is not square-free, then this expression is trivially $0$. Secondly, if $f$ is square-free and $\chi(f) = \chi(a)$, then this expression is equal to $d$. Finally, if $f$ is not square-free but $\chi(f) \neq \chi(a)$, then the expression is simply 
    $$\sum_{i=1}^d\xi^i = \frac{d}{k}\sum_{i=1}^k \xi^i $$
    where $\xi$ is a primitive $k$-th root of unity for some $k |d$. This sum is equal to $0$.

Thus, we see that
$$\frac{1}{d}\sum_{i=1}^d\sum_{\deg f \in [1, r)}\mu_q^2(f)\chi^i(f)\ov{\chi}^i(a) $$
counts the number of square-free polynomials $\deg f \in [1, r)$ with $\chi(f) = \chi(a)$. By the previous lemma, this is exactly the number of square-free polynomials $f \in a H_{d,F}$ with $\deg f \in [1,r)$.

We show this expression is greater than 0 for $(q,r) \in \Omega_0$. To do so, it is sufficient to show that the principal character dominates. That is, it suffices to show
\begin{align}\label{eq:principal_dominate}
    \sum_{\deg f \in[1, r)}\mu_q^2(j) > \bigg{|} \sum_{i=1}^{d-1}\sum_{\deg f \in [1, r)}\mu_q^2(f)\chi^i(f)\ov{\chi}^i(a) \bigg{|}. 
\end{align}

Using Corollary \ref{cor:SqFreeCount}, we have
\begin{align*}
    \sum_{\deg f \in [1,r)}\mu_q^2(j)
    = {(q-1)q^{r-1}}
\end{align*}
for $r>1$. 

We will split this discussion into two cases. We define the following subset of $\Omega_0$,
\begin{align*}
    \Omega_0'' =&\{(7, r):6 \leq r \leq 19 \}\cup\{(8, r):6 \leq r \leq 14 \}
\cup \{(9, r):5 \leq r \leq 10 \}\\\cup &\{(11, r):5 \leq r \leq 6 \}.
\end{align*}
Firstly we show (\ref{eq:principal_dominate}) is satisfied for $(q,r) \in \Omega_0''$, and then for $(q,r) \in \Omega_0\setminus \Omega_0''$. 

Applying (\ref{eq:min_sqfreesum}) and using $d < r$ we have

\begin{align*}
&\bigg{|} \sum_{i=1}^{d-1}\sum_{\deg f \in [1,r)}\mu_q^2(f)\chi^i(f)\ov{\chi}^i(a) \bigg{|}\\
& \hspace{5em} \leq (r-1)(q-1) \sum_{k=0}^{\fl{\frac{r-2}{2}}}q^k\sum_{t=1}^{r-2 k - 1}q^{t/2}\min\bigg{\{}q^{t/2}, \binom{r-1}{t}\bigg{\}}
\end{align*}
and thus it suffices to show 
\begin{align}\label{eq:coset_inequality}
q^{r-1} >(r-1) \sum_{k=0}^{\fl{\frac{r-2}{2}}}q^k\sum_{t=1}^{r-2 k - 1}q^{t/2}\min\bigg{\{}q^{t/2}, \binom{r-1}{t}\bigg{\}}.
\end{align}

It is simple to manually verify that this holds for all 31 pairs $(q,r) \in \Omega_0''$.

Secondly we consider $(q,r) \in \Omega_0\setminus \Omega_0''$. Using Corollary \ref{cor:SqFreeSum} and $d<r$ we obtain
\begin{align*}
&\bigg{|} \sum_{i=1}^{d-1}\sum_{\deg f \in [1,r)}\mu_q^2(f)\chi^i(f)\ov{\chi}^i(a) \bigg{|}\\
& \hspace{5em} \leq (r-1)(q-1)\frac{r}{2}q^{(r-1)/2}(2^{r-1}-1).
\end{align*}
and thus it suffices to show 
\begin{align}\label{eq:simple_coset_inequality}
q^{r-1} > (r-1)\frac{r}{2}q^{(r-1)/2}(2^{r-1}-1).
\end{align}
An appeal to basic calculus shows this holds for any pair $(q,r) \in \Omega_0\setminus \Omega_0''$. 
\end{proof}

We now fix some $F$ with $(q,r) \in \Omega_0$. For $d|q^{r}-1$ with $d < r$, let $C_{d,F}$ denote a set of square-free coset representatives for $H_{d,F}$ of degree less than $r$, as we know exist. There are of course $d$ elements in $C_{d,F}$. Also, let $S_{d,F}$ denote the set of irreducible polynomials that divide some member of $C_{d,F}$ and let 
$$S_F = \bigcup_{\substack{d|q^{r}-1 \\ d < r}} S_{d,F}.$$ Noting that a polynomial of degree less than $r$ has at most $r-1$ monic irreducible factors, we have $\#S_{d,F} < (q-1)dr $ and thus using $d< r$ we trivially have $\#S_F < (q-1)r^3.$

Now let $K$ denote the number of irreducible polynomials of degree less than $r$ not in $S_F$. So we have 
$$K = \sum_{k=1}^{r-1}\pi_q(k) - \#{S}_F > \sum_{k=1}^{r-1}\pi_q(k) - (q-1)r^3. $$

For a non-zero polynomial $s$ with $\deg s < r$ let $w(s)$ denote the number of unordered pairs of irreducible polynomials $u,v \not\in S_F$ with $\deg u, \deg v < r$ and  $uv \equiv s \Mod{F}$. Set 
$$\mathcal{A} = \{s : \deg s < r,~ s \neq 0,~ w(s) > Kq^{-r/2 + \slider}\}, ~ A = \#\mathcal{A}. $$

\begin{lem}
$$A > \frac{q^r-1}{r} + 2.$$
\end{lem}
\begin{proof}
Firstly, note that 
$$\sum_{\substack{\deg s < r \\ s \neq 0}}w(s) = \binom{K}{2} = \frac{1}{2}K(K-1). $$
Further, two distinct pairs counted by $f(s)$ cannot contain an irreducible in common, since all $u$ with $\deg u < r$ are distinct modulo $F$. Thus
$$f(s) \leq \frac{1}{2}K. $$
Since
$$\sum_{\substack{\deg s < r \\ s \not\in \mathcal{A}}}f(s) \leq (q^r-A)Kq^{-r/2 + \slider} $$
we can say 
$$\sum_{s \in \mathcal{A}}f(s) \geq \frac{1}{2}K(K-1)-(q^r-A)Kq^{-r/2 + \slider}. $$
Thus
\begin{align*}
    A
    &\geq \frac{2}{K}\sum_{s \in \mathcal{A}}f(s) 
    \geq K-1 - 2(q^r-A)q^{-r/2 + \slider}\\
    &\geq \sum_{k=1}^{r-1}\pi_q(k) - (q-1)r^3 -1 - 2q^{r/2 + \slider}.
\end{align*}
Applying Lemma \ref{lem:IrredCount} completes the proof.
\end{proof}

Next we need the following result due to Lev \cite[Theorem 2']{L1997}.
\begin{lem}\label{thm:Lev}
For a positive integer $B$ let $\mathcal{B} \subseteq \{1,...,B\}$ be a set of $n$ integers, and assume the positive integer $\kappa$ satisfies $\kappa \geq \frac{B-1}{n-2}-1$. Then there exists positive integers $d',k$ with $d' \leq \kappa$, $k \leq 2\kappa +1$ such that $k\mathcal{B}$ contains $B$ consecutive multiples of $d'$.  By $k\mathcal{B}$ we mean the set of integers that can be written as the sum of $k$ members of $\mathcal{B}$. 
\end{lem}
To use this, for any positive integer $k$ let $\mathcal{A}^k$ denote the set of $k$-fold products of members of $\mathcal{A}$. We have the following:

\begin{lem}\label{lem:Lev_appl}
There exists positive integers $d < r$ and $k < 2r +1$ such that $\mathcal{A}^k$ contains $H_{d,F}$. 
\end{lem}
\begin{proof}
Let $g$ be a generator of $\F_{q^r}^*$ and let $\mathcal{A}'$ denote the set of discrete logarithms of members of $\mathcal{A}$ to the base $g$. That is, $j \in \mathcal{A}'$ with $0 \leq j \leq q^r-1$ if and only if $g^j \Mod{F} \in \mathcal{A}$. Note that $\#\mathcal{A} = \#\mathcal{A}'$, since all polynomials of degree less than $r$ are distinct modulo $F$. 

Now we set $\kappa = (q^r-2)/(A-2)$, so by Lemma \ref{thm:Lev} we have that there are positive integers $d' \leq \kappa$ and $k \leq 2\kappa + 1$ such that $k\mathcal{A}'$ contains $q^{r}-1$ consecutive multiples of $d'$. Thus, reducing modulo $q^r-1$, the set $k\mathcal{A}'$ contains a subgroup of $\Z/(q^r-1)\Z$. The index of this subgroup is equal to $d := \gcd(d', q^r-1)$. Thus, $\mathcal{A}^k$ contains the subgroup $H_{d,F}$ of $\F_{q^r}^*$. 

We have 
\begin{align*}
    d \leq d' \leq \kappa = (q^r-2)/(A-2) \leq \frac{q^r-2}{(q^r-1)/r} < r
\end{align*}
and 
similarly $k < 2r + 1 $ as desired. 
\end{proof}

\begin{lem}\label{lem:h_rep}
For the subgroup $H_{d,F}$ produced in Lemma \ref{lem:Lev_appl}, each member of $H_{d,F}$ has a representation modulo $F$ as a square-free, $r-1$-smooth polynomial with all irreducible factors not in $S_F$. 
\end{lem}
\begin{proof}
Suppose $h \in H_{d,F} \subseteq \mathcal{A}^k$. Then we can write 
\begin{align}\label{eq:h_rep}
    h \equiv s_1...s_k \equiv (u_1v_1)...(u_kv_k) \Mod{F} 
\end{align}

with irreducible $u_i,v_i \not\in S_F$,  $s_i \equiv u_iv_i \Mod{F}$ and $\deg u_i, \deg v_i < r$. It is clear that this representation for $h$ has no irreducible factors of degree exceeding $r-1$, but it is not necessarily square-free. However, each $s_i$ has at least $q^{-r/2 + \slider}K$ representations $s_i \equiv u_iv_i \Mod{F}$, each of which consists of two new irreducible factors. If this bound is large enough in relation to $k$, then there will exist a representation of each $s_i$ so that the product in (\ref{eq:h_rep}) is square-free.

To understand how many representations are needed, firstly suppose that $k=2$. There at most $2(q-1)$ representations $s_2 \equiv u_2v_2 \Mod{F}$ such that the product $u_1v_1u_2v_2$ is not square-free. In each of these cases, either $u_2$ or $v_2$ would be any one of the $q-1$ unit multiples of either $u_1$ or $v_1$. Thus, $s_2$ having $2(q-1)+1$ representations is sufficient to ensure at least one of the products $u_1v_1u_2v_2$ is square-free.

Generalizing this to any $k$, we need at most $2(q-1)(k-1) + 1$ representations for each $s_i$ in order to ensure the product in (\ref{eq:h_rep}) is square-free. This is bounded by $2(q-1)(k-1) + 1 < 4r(q-1) + 1$. Thus, it suffices to show  $q^{-r/2 + \slider}K > 4r(q-1) + 1$. Using Lemma \ref{lem:IrredCount} it is enough to show
\begin{align}\label{eq:enough_representations}
 q^{-r/2+\slider}\bigg{(}\frac{q^{r}}{r} +2q^{r/2 + \slider} + 3\bigg{)}-4r(q-1)-1 > 0,
\end{align}
which holds for any $(q,r) \in \Omega_0 $. 
\end{proof}

\section{Proofs of main results}
\subsection{Proof of Theorem \ref{thm:SqFreeSmooth}}\label{subsec:mainproof}

For a fixed prime power $q \geq 7$, fix $F(X) \in \F_q[X]$ an irreducible polynomial of degree $r \geq 2$. 

We firstly consider the case $(q,r) \in \Omega_0$. Let $S_F$ and $\mathcal{A}$ be as defined in Subsection \ref{subsec:SqFreeCoset}, and furthermore choose $d,k$ such that $H_{d,F} \subseteq \mathcal{A}^k$ as given in Lemma \ref{lem:Lev_appl}. Now let $f$ be a non-zero residue class mod $F$ and choose $a \in {C}_{d,F}$ such that $f \in aH_{d,F}$. Then we can write $f = ah$ for some $h \in H_{d,F}$. We know by Lemma \ref{lem:h_rep} that $h$ has a square-free representative with all irreducible factors of degree less than $r$ not in $S_F$. Also, by Lemma \ref{lem:coset_rep} we know that $a$ is non-constant and square-free with all irreducible factors of degree less than $r$ all in $S_{F}$. Thus, $f$ has a non-constant, square-free, $r-1$-smooth representative modulo $F$.

Finally, we consider the case $(q,r) \in \Omega_1$. Recall,
\begin{align*}
\Omega_1 =& \{(q,r): q \geq 7, r \geq 2\} \setminus \Omega_0\\
=& \{(7,2), (7,3), (7,4), (7,5), (8,2), (8,3), (8,4), (8,5), (9,2),\\&~ (9,3), (9,4), (11,2), (11,3), (11,4), (13, 2) , (13, 3) , (16, 2) ,\\&~ (16, 3) , (17, 2) , (17, 3) ,
(19, 2) , (19, 3) , (23, 2) , (25, 2) , (27, 2) ,\\&~ (29, 2) , (31, 2) , (32, 2) , (37, 2) , (41, 2) ,
(43, 2) , (47, 2) , (49, 2) ,\\&~ (53, 2) , (59, 2) , (61, 2)\}. 
\end{align*}
In this case, a brute force approach is used. For every monic polynomial $f$ with $\deg f < r$, each of the polynomials $f + gF$ for $g \in \F_q[X]$ is considered until coming across one that is non-constant and divides 
\begin{align}\label{eq:irred_product}
  \prod_{\substack{\deg u < r \\ u~\text{monic, irreducible}}}u.   
\end{align}

Note that it suffices to only check monic polynomials $f$, since $f_1 \equiv f_2 \mod{F}$ if and only if  $f_1c \equiv f_2c \Mod{F}$ for any non-zero constant $c$, and $f_2c$ has the same irreducible factors as $f_2$. Similarly it also suffices to only check this for monic $F$, since $f_1 \equiv f_2 \Mod{F}$ if and only if $f_1 \equiv f_2 \Mod{cF}$ for any non-zero constant $c$. 

Using SageMath v. 8.8 \cite{Sage} to construct the polynomial rings, this computation has taken just under 4 hours and 45 minutes to check all the required cases, with the case $(q,r) = (8,5) $ taking the majority of the time at just under 4 hours and 28 minutes.

\subsection{Proof of Corollary \ref{cor:plus_n/d_sequence}}
Suppose $F$ is an irreducible polynomial of degree $r$ and the sequence thus far contains every irreducible polynomial of degree less than $r$. Let $N$ be the product of terms in the sequence so far. Let $\chi$ be the quadratic character on the finite field $\F_{q^r} \cong \F_q[X]/F(X)$.
If $\chi(-N) = 1$, then there exists some non-zero $a$ with $\deg a < r$ such that $a + N/a \equiv 0 \Mod{F}$. By Theorem \ref{thm:SqFreeSmooth} there exists some $h|N$ such that $h \equiv a \Mod{F}$, so we can choose $F$ as the next term in the sequence. 
If $\chi(-N) = -1$, we firstly show there exists some non-zero $a$ with $\deg a < r$ such that $\chi(a + N/a) = -1$. We follow the technique used to prove \cite[Lemma 3(i)]{BH2016}. Consider the sum 
\begin{align*}
    \sum_{y \in \F_{q^r}^*}\chi(y + N/y) = \sum_{y \in \F_{q^r}^*}\chi(y(y^2 + N)).
\end{align*}
Since $q$ is odd, $y(y^2 + N)$ has no repeated roots in $\F_{q^r}$ so 
$$\{(y,z) \in \F_{q^r}^2~:~z^2 = y(y^2 + N)\} $$
defines an elliptic curve (noting that this set is non-empty). Taking into account the point at infinity and using the Hasse bound (see \cite[Chapter 5 Theorem 1.1]{S2009}) we have 
$$1 + \sum_{y \in \F_{q^r}}\bigg{(}1 + \chi(y(y^2 + N)) \bigg{)} \leq 2q^{r/2} + q^r + 1. $$
Rearranging we obtain 
$$\sum_{y \in \F_{q^r}^*} \chi(y(y^2 + N)) \leq 2q^{r/2}  $$
which is less than $q^r-1$ for any odd prime power $q$ when $r \geq 2$. 

Thus there exists some non-zero $a$ with $\deg a < r$ such that $\chi(a + N/a) = -1$. By Theorem \ref{thm:SqFreeSmooth}, there exists some $h|N$ with $h \equiv a \mod{F}$, so of course $\chi(h + N/h) = -1$. and thus by multiplicativity there must exist some irreducible $g|h+N/h$ such that $\chi(g) = -1$. Choosing $g$ as the next term in the sequence, we thus have $\chi(-gN) = 1$. So by the first case considered, $F$ can now be taken as the next term in the sequence.  

\subsection{Proof of Corollary \ref{cor:plus_1_sequence}}
Firstly, note that all degree $1$ monic irreducible polynomials are produced in the expected order.

Next suppose $F$ is the smallest irreducible polynomial that has not been produced in the sequence (of degree at least 2) and all irreducible polynomials with smaller degree are in the sequence. Let $N$ denote the product of all elements in the sequence thus far. Then by Theorem \ref{thm:SqFreeSmooth} there exists some $h|N$ such that $h \equiv -1 \mod{F}$. Thus, $F$ appears as the next term in the sequence.

\section{Comments}

\subsection{Comments on Theorem \ref{thm:SqFreeSmooth}}
The proof of Theorem \ref{thm:SqFreeSmooth} shows that its conclusion holds for any $F$ such that the pair $(q,r)$ satisfies the inequalities (\ref{eq:lem_IrredCount_ineq}), (\ref{eq:enough_representations})  and either (\ref{eq:coset_inequality}) or (\ref{eq:simple_coset_inequality}). 

For $q=5$ we have (\ref{eq:lem_IrredCount_ineq}) and (\ref{eq:enough_representations}) hold for $r \geq 8$ and (\ref{eq:simple_coset_inequality}) holds for $r \geq 72$. By manually checking that (\ref{eq:coset_inequality}) holds for $13 \leq r \leq 71$, we see that the conclusion of Theorem \ref{thm:SqFreeSmooth} holds for $q=5$ when $r \geq 13$. The remaining cases $r \leq 12$ could be verified as part of a large-scale computing project. 

The smaller prime powers may be more difficult to handle. For $q=4$, (\ref{eq:lem_IrredCount_ineq}) and (\ref{eq:enough_representations}) hold for $r \geq 9$ and it appears that (\ref{eq:coset_inequality}) holds for $r \geq 22$, but (\ref{eq:simple_coset_inequality}) never holds. Similarly for $q=3$, (\ref{eq:lem_IrredCount_ineq}) and (\ref{eq:enough_representations}) hold for $r \geq 12$ and it appears that (\ref{eq:coset_inequality}) holds for $r \geq 49$, but (\ref{eq:simple_coset_inequality}) never holds.  Unfortunately for $q=2$, it appears that neither (\ref{eq:coset_inequality}) nor (\ref{eq:simple_coset_inequality}) ever hold.

\subsection{Comments on Conjecture \ref{conj:allq}}\label{subsec:conjecture}
Using the SageMath program written in order to finish the proof of Theorem \ref{thm:SqFreeSmooth}, the conclusion of Theorem \ref{thm:SqFreeSmooth} has been tested for all $F$ with $q,r \leq 5$ using the same procedure as described in subsection \ref{subsec:mainproof}. That is, for each monic irreducible polynomial $F$ of degree $r$ over $\F_q$ such that $q \leq 5$ and $2 \leq r \leq 5$, we have manually checked whether every residue class modulo $F$ has a non-constant representative dividing (\ref{eq:irred_product}). We call such a representative a `suitable' representative. 

This computation has taken just under $118$ seconds, with only two polynomials found not to satisfy the conclusion of Theorem \ref{thm:SqFreeSmooth}. Both in the case of $q=2$ we have that $1$ does not have a suitable representative modulo $x^3+x+1$ nor modulo $x^3 + x^2 + 1$. Although this is a small sample, this has lead the author to believe that Conjecture \ref{conj:allq} holds.

\subsection{Comments on Corollary \ref{cor:plus_n/d_sequence}}
In the construction of the sequence in Corollary \ref{cor:plus_n/d_sequence}, it is natural to ask whether starting with all degree $1$ polynomials in the sequence is necessary. That is, for example, whether the sequence can be chosen to contain every degree $1$ polynomial, starting with only $X$ as the first term in the sequence. 

Even to computationally check this for a given prime power is difficult, because each term in the sequence must be chosen from a number of candidates. Again using SageMath, every possibility for the first $7$ terms in the sequence in the case of $q=3$ has been computed, with a run-time of under 15 minutes. There are a total of 1397132 such possibilities for these first 7 terms, none of which contain more than 3 degree 1 polynomials. 

\subsection{Comments on Corollary \ref{cor:plus_1_sequence}} Corollary $\ref{cor:plus_1_sequence}$ is equivalent to the statement that $-1$ always has a non-constant, square-free, $r-1$-smooth representative modulo any irreducible polynomial $F$ of degree $r \geq 2$ over $\F_q$, for $q \geq 7$. Conjecture \ref{conj:allq} would imply that this holds for $q \geq 3$. 

Subsection \ref{subsec:conjecture} described that in the case of $q=2$, $-1 $ does not have such a representative modulo $x^3+x+1$ nor modulo $x^3 + x^2 +1$. Since these constitute all irreducible polynomials of degree $3$ over $\F_2$, the sequence described in Corollary \ref{cor:plus_1_sequence} will not produce all irreducible polynomials in order. Computing the first 6 terms in this sequence gives
\begin{align*}
    &x, x + 1, x^2 + x + 1, x^4 + x + 1, x^3 + x + 1, x^3 + x^2 + 1.
\end{align*}
Although the degree $3$ polynomials appeared out of order, they are nonetheless both produced in the sequence. Thus, Conjecture \ref{conj:allq} would imply that after this point all irreducible polynomials are produced in order. 

\subsection{Comments on Lemma \ref{lem:MonicSum}} 
We have taken advantage of a very  explicit  character sum estimate given in~\cite{H2020}.
 For a given $q$ this seems to be the best known estimate for small values of $t$ and $r$, but asymptotically the bound given in~\cite[Theorem~1]{BLL2017} is better. 
We use this opportunity to observe that the condition on $n$ and $q$ in~\cite[Theorem~1]{BLL2017} is stated incorrectly, and the inequality should read
$$
\frac{\log \log n}{\log n} \le \frac{1}{\log q}.
$$
Unfortunately this invalidates the comparison between~\cite[Theorem~1]{BLL2017} 
and~\cite[Theorem~1.3]{H2020} given by Han~\cite[Remark~1.4]{H2020}.

 \section*{Acknowledgements}
The author would like to thank Igor Shparlinski for many helpful ideas, comments and corrections. During the preparation of this work, the author was supported by an Australian Government Research Training Program (RTP) Scholarship. 

\bibliographystyle{plain}


\end{document}